\documentclass[reqno,10pt]{amsart}

\usepackage{mathrsfs}
\usepackage{color,hyperref}
\usepackage{amsmath}
\usepackage{amssymb}

\newtheorem{thm}{Theorem}[section]

\long\def\oo#1{}

\newtheorem{theorem}[thm]{Theorem}
\newtheorem{lm}[thm]{Lemma}

\newtheorem{pro}[thm]{Proposition}
\newtheorem{defi}[thm]{Definition}
\newtheorem{ex}[thm]{Example}

\newtheorem{conjecture}[thm]{Conjecture}

\usepackage[numbers]{natbib}

\begin{document}

\title[Semifields and a theorem of Abhyankar]{Semifields and a theorem of Abhyankar} 
\author{V\'\i t\v ezslav Kala}
\address{Faculty of Mathematics and Physics, Department of Algebra, Charles University, Sokolovsk\' a 83,
18600 Praha 8, Czech Republic}
\address{Mathematisches Institut, Georg-August Universit\" at G\" ottingen, Bunsenstra\ss e 3-5, 37073 G\" ottingen, 
Germany}
\email{vita.kala@gmail.com}


\date{\today}


\begin{abstract}
Abhyankar \cite{abh} proved that every field of finite transcendence degree over $\mathbb{Q}$ or over a finite field is a homomorphic image of a subring of the ring of polynomials $\mathbb{Z}[T_1, \dots, T_n]$ (for some $n$ depending on the field). 
We conjecture that his result can not be substantially strengthened and show that our conjecture implies a well-known conjecture on the
additive idempotence of semifields that are finitely generated as semirings \cite{BHJK}, \cite{JKK}.
\end{abstract}

\maketitle

\section{Introduction}

A classical fact says that if a field is finitely generated as a ring, then it is finite -- in other words, no infinite field is a homomorphic image of
the ring of polynomials $\mathbb Z[T_1, \dots, T_n]$.
Surprisingly, Abhyankar in 2011 proved the following theorem saying that this is not true when we consider fields as factors of subrings of $\mathbb Z[T_1, \dots, T_n]$. 

\begin{theorem}[\cite{abh}, Proposition 1.2]\label{abh thm}
Let $F$ be a field of finite transcendence degree over $\mathbb{Q}$ or over a finite field $\mathbb F_q$.
Then there is a ring $B\subset A=\mathbb{Z}[T_1, \dots, T_n]$ (for a suitable $n$) and a maximal ideal $I$ of $B$ such that $F\simeq B/I$.
\end{theorem}

For example, there is a subring $B\subset \mathbb Z[T_1, T_2]$ and an epimorphism $B\twoheadrightarrow \mathbb Q$, described in Section \ref{abh construction} below.

Rings $B$ that arise from Abhyankar's construction have several unusual properties. 
The goal of this short note is to explore them and to study the (im)possibility of a generalization and a connection to the theory of semirings.

\

Just to recall, by a \textit{semiring} $S(+, \cdot)$ we here mean a set $S$ with addition $+$ and multiplication $\cdot$ that are both commutative and associative, and such that multiplication distributes over addition. A semiring $S(+, \cdot)$ is a \textit{semifield} if moreover $S(\cdot)$ is a group
(such a structure is also occasionally called a parasemifield \cite{BHJK}, \cite{JKK}; note that unlike our definition, sometimes a semifield is defined to have a zero element).
A semiring is  \textit{additively idempotent} if $a+a=a$ for all $a\in S$.

Semirings are a very natural generalization of rings that has been widely studied, not only from purely algebraic perspective, but also for their applications in cryptography, dequantization, tropical mathematics, non-commutative geometry, and the connection to logic via MV-algebras and lattice-ordered groups 
\cite{BHJK}, \cite{BCM},  \cite{DNG}, \cite{NL}, \cite{G}, \cite{Le}, \cite{litvinov}, \cite{monico}, \cite{M}, \cite{W}, \cite{WW}, and \cite{zumbragel}. We refer the reader to the aforementioned works for further history and references.

\

Particularly important and interesting are the problems of classifying various classes of semirings.
In order to classify all ideal-simple (commutative) semirings \cite{BHJK}, one needs to study the following conjecture that extends the above-mentioned classical result on fields that are finitely generated as rings:

\begin{conjecture}[\cite{BHJK}]\label{main-conj}
Every semifield which is finitely generated as a semiring is additively idempotent.
\end{conjecture}

A lot of progress has already been made on Conjecture \ref{main-conj}: building on the results from \cite{KK}, \cite{KKK}, the conjecture was 
proved for the case of two generators in \cite{JKK}. Recently, additively idempotent semifields that are finitely generated as semirings were classified \cite{Ka} using their correspondence with lattice-ordered groups. This result then provides new tools for attacking the conjecture in the case of more generators \cite{KK2}. 
Some partial results have also been obtained for a generalization of this problem to divisible semirings (instead of semifields) \cite{KeKo}, \cite{KoLa}, \cite{Ko}.

\

In this short note we consider the connection between semirings, semifields and rings. The basic construction is that of the difference ring 
$S-S$ of a semiring $S$. If $S$ is additively cancellative (i.e., $a+c=b+c$ implies $a=b$ for all $a, b, c\in S$), we can define $S-S$ as the set of all (formal) differences $s-t$ of elements of $S$ with addition and multiplication naturally extended from the semiring. 
The fact that $S$ is additively cancellative ensures that $S-S$ is well-defined and that $S\subset S-S$.

The situation is much more complicated in the case of semifields that might violate Conjecture \ref{main-conj} -- it is easy to show that they can not be additively cancellative. 
In fact, more generally assume that we have a finitely generated semiring $S$ that contains the semifield of positive rational numbers $\mathbb Q^+$. Then we know that $S$ is not additively cancellative \cite[Proposition 1.18]{KKK}, and so the difference ring $S-S$ is not well-defined. One could instead consider the Grothendieck ring $G(S)$, but we still would not have $S\subset G(S)$.

However, we can proceed somewhat less directly: $S$ is finitely generated, and so it is a factor of the polynomial semiring $\mathbb N[T_1, \dots, T_n]$. Let $\varphi: \mathbb N[T_1, \dots, T_n]\twoheadrightarrow S$ be the defining epimorphism. $\mathbb Q^+\subset S$ is additively cancellative, and so we can consider
the difference ring $B:=\varphi^{-1}(\mathbb Q^+)-\varphi^{-1}(\mathbb Q^+)\subset \mathbb Z[T_1, \dots, T_n]$ and extend $\varphi$ to a surjection $B\twoheadrightarrow \mathbb Q$. The existence of such a $B$ is purely a ring-theoretic statement. If such a $B$ did not exist, we would have obtained a contradiction with the existence of $S$. 
However, precisely such rings were constructed by Abhyankar \cite{abh}, and so this naive approach fails.

Despite this attempt not working, there are finer ways of translating the problem to the setting of rings. As our main result, we show in Theorem \ref{3.3.1} that the existence of a semifield which is finitely generated as a semiring and not additively idempotent (i.e., one violating the Conjecture \ref{main-conj}) implies the existence of a ring which would contradict the following conjecture. 

\begin{conjecture}\label{ring conj}
The following situation can not happen:

Let $A=\mathbb{Z}[T_1, \dots, T_n]$. There exists a subring $B$ of $A$ and an ideal $I$ of $B$ such that

a) $QF(B)=\mathbb{Q}(T_1, \dots, T_n)$ ($QF$ denotes the quotient field).

b) $IA=A$.

c) Let \ $A^+=\mathbb{N}[T_1, \dots, T_n]$. Then there is $\ell_0\in A^+$ such that for all $a\in A^+$ and all $\ell=\ell_0+a\in A$, we have 
if $h(\ell)=0$ for some $h(X)\in B[X]$, then $h(X)\in I[X]$.

d) $\mathbb{Q}\subset B/I$.
\end{conjecture}

The significance of the conditions above will become clear from the proof of Theorem \ref{3.3.1} below. Let us only remark here that Abhyankar's rings from Theorem \ref{abh thm} satisfy a), b), and d), but not c).

Condition c) in the conjecture does not look very natural from ring-theoretic point of view. However, we are not even aware of whether the conjecture holds when we replace c) by

\textit{c') For all $\ell\in A$ we have that if $h(\ell)=0$ for $h(X)\in B[X]$, then $h(X)\in I[X]$.}

It seems plausible to expect that this Conjecture \ref{ring conj} is in fact equivalent to the Conjecture \ref{main-conj}. This is certainly an interesting direction for future research.


\section{Abhyankar's construction}\label{abh construction}

Abhyankar's construction and the proof of Theorem \ref{abh thm} are fairly involved, being based on the notion of a blow-up from algebraic geometry, which is also used in the resolution of singularities of algebraic varieties. Hence we refer the reader to the original paper for details; let us just illustrate  the construction with an explicit example of a ring with a surjection onto $F=\mathbb Q$, in which case we can take $n=2$.

\begin{ex}
We have $A=\mathbb{Z}[T_1, T_2]$. Define $B=\mathbb{Z}[f_2, f_3, \dots]$ and a homomorphism $\varphi: B \twoheadrightarrow \mathbb Q$ by

$f_2=T_1 T_2$

$f_{n+1}=(nf_n-1) T_2$

$\varphi(f_n)= \frac 1{n}$

Then simply let $I=\ker \varphi$ and we have $\mathbb Q\simeq B/I=\mathrm{Im} \varphi$.
\end{ex}

Note that not only the condition d), but also the conditions a) and b) of Conjecture \ref{ring conj} are satisfied in this case. However, c) is far from being true. Similarly, this example can be modified to give counterexamples to versions of Conjecture \ref{ring conj} with other conditions omitted.

It's also trivial to observe that the homomorphism $\varphi$ can't be extended to $A$: If $\psi: A \twoheadrightarrow \mathbb Q$ was such an extension, then we would have $$\frac 1{n+1}=\psi(f_{n+1})=\psi(nf_n-1) \psi(T_2)=(n\psi(f_n)-1)\psi(T_2)=0,$$
a contradiction.

\

A similar construction works generally (for other fields of finite transcendence degree over $\mathbb Q$ or $\mathbb F_p$). However, the proof given in \cite{abh} is not very elementary, it uses local rings, blow-ups, etc. It is an interesting question whether it is possible to give an elementary proof.

Details are in the first 3 sections of \cite{abh}, especially see Proposition 3.1.


\section{Outline of the proof}\label{ch3 proof}

In this section we prove the following Theorem \ref{3.3.1}. Along the way we collect and review various useful properties of semifields extending and generalizing \cite{JKK}. 

\begin{theorem}\label{3.3.1}
Conjecture \ref{ring conj} implies Conjecture \ref{main-conj}.
\end{theorem}

Throughout the rest of this section we will assume that Conjecture \ref{main-conj} does not hold and construct a counterexample to Conjecture \ref{ring conj}. We will use the following notation:

\begin{defi}
Let $A=\mathbb{Z}[T_1, \dots, T_n]$, $A^+=\mathbb{N}[T_1, \dots, T_n]$. 

Let $\varphi: A^+ \twoheadrightarrow S$ be a surjection onto a semifield $S$, which is finitely generated as a semiring and not additively idempotent. We then have $\mathbb{Q^+}\subset S$.

On $S$ we have the natural preorder defined by $a\leq b$ if $b=a+c$ for some $c\in S$.
Consider $P=\{s\in S | \exists q, r\in\mathbb{Q^+}: q\leq s\leq r\}$.  By \cite[Proposition 3.11]{KKK}, $P$ is an additively cancellative semifield.

Let $B^+ = \varphi^{-1} (P)\subset A^+$. Since $P$ is additively cancellative, we can take the difference ring $P-P$ and extend $\varphi$ to $\varphi^\pm: B=B^+-B^+\twoheadrightarrow P-P$ ($\varphi^\pm$ is onto). 

Take $I=\ker\varphi^\pm$, an ideal of $B$. Notice that we have $\mathbb{Q}\subset B/I$.
\end{defi}

This attaches to the semifield $S$ the objects from Conjecture \ref{ring conj}. Now we just need to check they have all the required properties. Condition d) was already verified above and b) follows easily:

\begin{pro}
Let $f, g\in A^+$. If $\varphi (f)=\varphi (g)$ then $f-g\in IA$.

Therefore $IA=A$.
\end{pro}

\begin{proof}
$S$ is a semifield, and so there is $h\in A^+$ such that $\varphi(gh)=1$. Then also $\varphi(fh)=1$ and $f-g=(1-gh)f+(fh-1)g\in IA$.

Now by \cite[Proposition 3.14]{KKK}, there is an $\ell\in A^+$ such that  $\varphi (\ell)+1=\varphi (\ell)$. Then $\varphi (\ell+1)=\varphi (\ell)$, and so by the first part, $1=(\ell+1)-\ell\in IA$.

Since $IA$ is an ideal in $A$, we have $IA=A$.
\end{proof}

To prove a) and c), 
we need to consider a certain subsemiring $Q$ of $S$ and its properties:

\begin{thm}\label{properties of Q}
{\rm (Properties of $Q$, summary of various statements from \cite{JKK} and \cite{KKK})}

Let $Q=\{s\in S | \exists q\in\mathbb{Q^+}: s\leq q\}$. Then:

a) If $a_1+\dots +a_n\in Q$, then $a_i\in Q$ for each $i$. Hence $\varphi^{-1} (A)$ is generated by monomials as an additive semigroup.

b) $a\in Q$ if and only if $qA\in Q$ for all $q\in\mathbb Q^+$.

c) $Q+\mathbb Q^+=P$

Let $\mathcal C$ be the ``cone" of $Q$, $\mathcal C=\{(u_1, \dots, u_n)\in\mathbb N_0^n | \varphi(T_1^{u_1}\cdots T_n^{u_n})\in Q\}\subset \mathbb N_0^n$. (If $u=(u_1, \dots, u_n)\in\mathbb N_0^n$, we sometimes write just $T^u:=T_1^{u_1}\cdots T_n^{u_n}$.)

d) $\mathcal C$ is a pure semigroup, i.e., if $a\in\mathcal C$, $k\in\mathbb N$ and $a/k\in\mathbb N_0^n$, then $a/k\in\mathcal C$.
\end{thm}

\begin{defi}
Let $\dim\ \mathcal C$ denote the smallest dimension of a linear subspace of $\mathbb{R}^n$ which contains $\mathcal C$.
\end{defi}

\begin{lm}\label{prop cone}
a) $\dim\ \mathcal C=n$

b) There exists $u\in \mathcal C$ s.t. $u+(1,0,\dots, 0)$, $u+(0,1,0,\dots, 0), \dots,$ $u+(0,\dots, 0, 1)$ $\in \mathcal C$.
\end{lm}

\begin{proof}
Take $f=T_1+\dots +T_n$ and let $g\in A^+$ be such that $\varphi (fg)=1$. Write $g=\sum c_i T^{u^{(i)}}$, where $c_i\in\mathbb N, u^{(i)}\in\mathbb N_0^n$. 
Thus $\varphi (fg)\in Q$, and so (by Theorem \ref{properties of Q} a)) $u^{(1)}+(1,0,\dots, 0)$, $u^{(1)}$ $+(0,1,0,\dots, 0), \dots, u^{(1)}+(0,\dots, 0, 1)$ are $n$ linearly independent vectors in $\mathcal C$. Hence $\mathrm{dim}\ \mathcal C=n$.
\end{proof}

\begin{lm}\label{P elem}
If $u\in \mathcal C$, then $\varphi (1+T^u)\in P$.
\end{lm}

\begin{proof}
This follows just from Theorem \ref{properties of Q} c).
\end{proof}

Now we are ready to show that the condition a) of the Conjecture \ref{ring conj} is satisfied.

\begin{pro}\label{qf}
$QF(B)=\mathbb{Q}(T_1, \dots, T_n)$
\end{pro}

\begin{proof}
By Lemma \ref{P elem} we have $1+T^u\in B^+$ for all $u\in \mathcal C$. Since also $1\in B^+$, we see that $T^u\in B$ for all $u\in \mathcal C$.

By Lemma \ref{prop cone} b) we have $u+(1,0,\dots, 0), u+(0,1,0,\dots, 0), \dots,$ $u+(0,\dots, 0, 1)$ $\in \mathcal C$ for some $u$. Thus $T^{\mathrm{any\ of\ these\ vectors}}\in B$. Since also $T^u\in B$, we see that $QF(B)\ni T_1=T^{(1,0,\dots, 0)}=\frac{T^{u+(1,0,\dots, 0)}}{T^u}$. Similarly we see that all other generators $T_i\in QF(B)$, concluding the proof.
\end{proof}

It remains only to show the condition c).

\begin{pro}
There is $\ell_0\in A^+$ such that for all $a\in A^+$ and all $\ell=\ell_0+a\in A$ we have:  if $\ell=f/g$ for $f, g\in B$, then $f, g\in I$.

More generally, let $\ell$ be as above. If $h(\ell)=0$ for $h(X)\in B[X]$, then $h(X)\in I[X]$.
\end{pro}

\begin{proof}
Consider $L=\{\ell\in A^+ | \varphi(\ell)+1=\varphi(\ell)\}$. $L$ is non-empty by \cite[Proposition 3.14]{KKK}. Clearly $L+A^+\subset L$. Hence if we take $\ell_0\in L$, then every $\ell=\ell_0+a$ (for $a\in A^+$) lies in $L$. Also note that $\varphi(\ell)+p=\varphi(\ell)$ for all $\ell\in L$ and $p\in P$.

Take now any $\ell\in L$ and assume that $\ell=f/g$ for some $f, g\in B$ (note that by Proposition \ref{qf} such $f, g$ exist for each $\ell$). We have $f, g\in B$, so there are $f_1, f_2, g_1, g_2\in B^+$ such that $f=f_1-f_2, g= g_1-g_2$.
Then $\ell g_1+f_2=\ell g_2+f_1$. 

Since $\ell\in L$ and $\varphi (f_i), \varphi (g_i)\in P$, we have $\varphi(\ell g_1+f_2)=\varphi(\ell g_1)$ and $\varphi(\ell g_2+f_1)=\varphi(\ell g_2)$. Hence $\varphi(\ell g_1)=\varphi(\ell g_2)$, and so $\varphi(g_1)=\varphi(g_2)$. But then also $\varphi(f_1)=\varphi(f_2)$.

Hence $f, g\in I$ as needed.

Proof of the second part is essentially the same. Write $h=f-g$ with $f, g\in B^+[X]$. Then $\varphi(f(\ell))=\varphi(g(\ell))$. As before, this implies that we can remove the constant terms from the equality and divide by $\ell$. We get a new equality of polynomials of the same form and we can proceed inductively till we get that the $\varphi(f), \varphi(g)$ have the same degree and equal leading coefficients. Then we can again proceed inductively and show that in fact all coefficients of $\varphi(f), \varphi(g)$ are equal, which means that $\varphi(f)-\varphi(g)=0$, and so $h=f-g\in I[X]$.
\end{proof}
\def\bibname{Bibliography}


\begin{thebibliography}{99}
\addcontentsline{toc}{chapter}{\bibname}



\bibitem{abh}
\rm S. S. Abhyankar, 
\it Pillars and towers of quadratic transformations, 
\rm Proc. AMS {\bf 139} (2011), pp. 3067--3082.


\bibitem{BHJK}
\rm R.~El~Bashir, J.~Hurt, A.~Jan\v ca\v r\'\i k and T.~Kepka,
\it Simple commutative semirings,
\rm Journal of Algebra {\bf 236} (2001), 277--306.

\bibitem{BCM}
\rm M. Busaniche, L. Cabrer, and D. Mundici, 
\it Confluence and combinatorics in finitely generated unital lattice-ordered abelian groups, 
\rm Forum Math. {\bf 24} (2012), 253-271.

\bibitem{DNG}
\rm A. Di Nola and B. Gerla, 
\it Algebras of Lukasiewicz's logic and their semiring reducts,
\rm Contemp. Math. {\bf 377} (2005), 131--144.

\bibitem{NL} 
\rm A. Di Nola and A. Lettieri, 
\it Perfect MV-Algebras are Categorically Equivalent to Abelian $\ell$-Groups, 
\rm Studia Logica, {\bf 53} (1994), 417--432.

\bibitem{G}
\rm J. S. Golan, 
\it Semirings and Their Applications, 
\rm Kluwer Academic, Dordrecht {\bf 1999}.

\bibitem{JKK}
\rm J. Je\v zek, V. Kala, T. Kepka, 
\it Finitely generated algebraic structures with various divisibility conditions, 
\rm Forum Math. {\bf 24} (2012), 379--397.

\bibitem{Ka}
\rm V.~Kala,
\it Lattice-ordered groups finitely generated as semirings,
\rm J. Commut. Alg., 16 pp., to appear, \texttt{arxiv:1502.01651}.

\bibitem{KK}
\rm V.~Kala, T.~Kepka,
\it A note on finitely generated ideal-simple commutative semirings,
\rm Commentationes Math. Univ. Carolinae {\bf 49} (2008), 1--9.

\bibitem{KKK} 
\rm V. Kala, T. Kepka, M. Korbel\'a\v r, 
\it Notes on commutative parasemifields, 
\rm Comment. Math. Univ. Carolin. {\bf 50} (2009), 521-533.

\bibitem{KK2} 
\rm V. Kala,  M. Korbel\'a\v r, 
\it Commutative semifields finitely generated as semirings are idempotent, 
\rm in preparation.

\bibitem{KeKo}
T. Kepka, M. Korbel\'{a}\v{r}, \emph{Conjectures on additively
divisible commutative semirings}, Mathematica Slovaca, to appear, \texttt{arxiv:1401.2836}.

\bibitem{KoLa} M.~Korbel\'{a}\v{r}, G.~Landsmann, \emph{One-generated semirings and
additive divisibility}, J. Algebra Appl. \textbf{16} (2017), 22 pp.,
DOI: 10.1142/S0219498817500384.

\bibitem{Ko} M.~Korbel\'{a}\v{r}, \emph{Torsion and divisibility in finitely
generated commutative semirings}, Semigroup Forum, to appear. DOI:
10.1007/s00233-016-9827-4


\bibitem{Le} E. Leichtnam, \emph{A classification of the commutative Banach perfect semi-fields of characteristic 1. Applications}, \texttt{arxiv:1606.00328}.

\bibitem{litvinov}
\rm G. L. Litvinov,
\it The Maslov dequantization, idempotent and tropical mathematics: a brief introduction,
\rm \texttt{arXiv:math/0507014}.

\bibitem{monico}
\rm C. J. Monico,
\it Semirings and semigroup actions in public-key cryptography,
\rm PhD Thesis, University of Notre Dame, USA, {\bf 2002}, vi+61 pp.

\bibitem{M}
\rm D. Mundici, 
\it Interpretation of AF $C^*$-algebras in \L ukasiewicz sentential calculus, 
\rm J. Funct. Anal., {\bf 65} (1986), 15--63.



\bibitem{W} 
\rm H. J. Weinert, 
\it Uber Halbringe und Halbkorper. I., 
\rm Acta Math. Acad. Sci. Hungar. {\bf 13} (1962), 365-378. 



\bibitem{WW}
\rm H.~J.~Weinert and R.~Wiegandt,
\it On the structure of semifields and lattice-ordered groups,
\rm Period. Math. Hungar. {\bf 32} (1996), 147--162.


\bibitem{zumbragel}
\rm J. Zumbr\" agel,
\it Public-key cryptography based on simple semirings,
\rm PhD Thesis, Universit\" at Z\" urich, Switzerland, {\bf 2008}, x+99 pp.



\end{thebibliography}
\end{document}